 \documentclass[11 pt,a4paper]{article}
 
 \usepackage{amsmath,amscd}
\usepackage{amstext}
\usepackage{amssymb,epsfig}
\usepackage{amsthm}
\usepackage{amsfonts}
\usepackage{graphicx}
\usepackage{psfrag}
\usepackage[all]{xy}

\newcommand{\N}{\mathbb{N}}
\newcommand{\Q}{\mathbb{Q}}

\newcommand{\Z}{\mathbb{Z}}

\newcommand{\bw}{\mathbb{S}}
\newcommand{\pal}{\mathbb{P}}

\renewcommand\Tilde[1]{{#1}\string~}

\newcommand{\vect}[1]{{\protect\boldsymbol{\mathrm{#1}}}}
\newcommand{\resp}[1]{(resp.~#1)}

\newtheorem{theorem}{Theorem}
\newtheorem{lemma}[theorem]{Lemma}
\newtheorem{corollary}[theorem]{Corollary}
\newtheorem{proposition}[theorem]{Proposition}
\newtheorem{definition}{Definition}
\newtheorem{example}{Example}

\let\dsp\displaystyle

\let\iff\Leftrightarrow

\def\dfrac#1#2{\frac{\dsp#1}{\dsp#2}}
\let\capori\cap
\def\cap{\mathrel\capori}
\let\cupori\cup
\def\cup{\mathrel\cupori}

\title{On the Number of Balanced Words of Given Length and Height over a Two-Letter Alphabet}

\author{Nicolas B\' edaride \footnote{Université Aix-Marseille III, Avenue de  l'Escadrille Normandie-Niémen, 13397 Marseille Cedex 20, France,{\ttfamily nicolas.bedaride@univ-cezanne.fr}}, 
Eric Domenjoud, 
  Damien Jamet,
  Jean-Luc R\'emy\footnote{Loria - Université Nancy 1 - CNRS, Campus Scientifique,
BP 239, 54506 Vand{\oe}uvre-les-Nancy, France,
\ttfamily{\{eric.domenjoud, damien.jamet, jean-luc.remy\}@loria.fr}}}

\begin{document}

\maketitle
\begin{abstract}
We exhibit a recurrence on the number of discrete line segments joining two integer points in the plane
using an encoding of such segments as balanced words of given length and height over the
two-letter alphabet $\{0,1\}$. We give generating functions and study the asymptotic
behaviour. As a particular case, we focus on the symmetrical discrete
segments which are encoded by balanced palindromes.
\end{abstract}

\section{Introduction}\label{intro}
The aim of this paper is to study some properties of discrete lines by using combinatorics on words.
The first investigations on discrete lines are dated back to
J.~Bernoulli\cite{Ber1772}, E.B.~Christoffel~\cite{Chr1875}, A.~Markoff~\cite{Mar1882} and more recently to G.A.~Hedlund and H.~Morse~\cite{MorHed1940} who introduced the terminology of \emph{Sturmian sequences}, for the 
ones defined on a two-letter alphabet and coding lines with irrational slope.  These works
gave the first theoretical framework for discrete lines.  A sequence $u \in \{0,1\}^\N$ is
Sturmian if and only if it is \emph{balanced} and not-eventually periodic. From the 70's,
H.~Freeman~\cite{Fre1974}, A.~Rosenfeld~\cite{Ros1974b} and S.~Hung~\cite{Hun1985} extended these investigations to lines with rational slope and studied \emph{discrete  segments}. 
In~\cite{Rev1991}, J.-P.~Reveill\`es defined arithmetic discrete lines as sets of
integer points between two parallel Euclidean lines. There are two sort of arithmetic discrete lines, the \emph{naive} and the \emph{standard} one. 

There exists a direct relation between naive \resp{standard} discrete arithmetic lines and
Sturmian sequences. Indeed, given a Sturmian sequence $u \in \{0,1\}^\N$, if one
associates the letters $0$ and $1$ with a shifting along the vector $\vect{e_1}$ and
$\vect{e_2}$ \resp{the vectors $\vect{e_1}$ and $\vect{e_1}+\vect{e_2}$} respectively,
then, the vertices of the obtained broken line are the ones of a naive arithmetic discrete
line \resp{a standard arithmetic discrete line} with the same slope  (see Figure~\ref{fig::fig_ex_seg}). 

Let $s : \N \longmapsto \N$ be the map defined by:
\begin{equation*}
\begin{array}{ccccl}
s & : & \N & \longrightarrow & \N \\
  &   & L  & \mapsto         & \#\{w \in \{0,1\}^L, \, w \text{ is balanced}\},
\end{array}
\end{equation*}
where $\#E$ denotes the cardinal of the set $E$. 
In other words, given $L \in \N$, $s(L)$ is the number of balanced words of length $L$, or
equivalently, the number of discrete segments of any slope $\alpha \in [0,1]$ of length
$L$. In~\cite{Lip1982}, it is proved that 
\begin{equation*}
s(L) = 1 + \sum_{i=1}^L{(L-i+1)\varphi(i)},
\end{equation*}
 where $\varphi$ is Euler's totient function, that is, $\varphi(n)$ is the number of positive
 integers smaller that $n$ and coprime with $n$. Alternative proofs of this result can be
 found in~\cite{Mig1991,BerPoc1993,CASSAIGNE:2002:HAL-00133481:1, BerLav1988}. 
 
In~\cite{LucLuc2005,LucLuc2006}, de Luca and De Luca investigated the number $p(L)$ of balanced palindrome
words of length $L \in \N$, that is the balanced words coding a symmetrical discrete
segments of length $L$. They proved
\begin{equation*}
p(L) = 1 + \sum_{i=0}^{\lceil L/2 \rceil - 1}{\varphi(L-2i)}.
\end{equation*}

In the present work, we investigate the following question. Given two integer points of
$\Z^2$ (also called \emph{pixels} in the discrete geometry literature~\cite{ChaMon1991}),
how many naive discrete segments link these points (see Figure~\ref{fig::fig_ex_seg})? In
other words, given $L \in \N$ and $h \in \N$, how much is $s(L,h)=\# \{ w \in \{0,1\}^L,
\vert w \vert_1=h \text{ and } w \text{ balanced}\}$? We exhibit a recurrence relation
on $s(L,h)$ and generating functions and we study the asymptotic behaviour of the maps
$s$. After this, we focus on the number $p(L,h)$ of balanced palindromes of given length and height for which we also exhibit a recurrence relation and a generating function.

We are interested in these formulas to have a better understanding of the space of Sturmian sequences. Indeed the main combinatorial properties of theses sequences can be seen in similar formulas. For example the formula of $s(L)$ is deeply related to the number of bispecial words of length $L$, see \cite{CASSAIGNE:2002:HAL-00133481:1}. One main objective is to generalize these formulas to dimension two in way to understand the combinatorics structure of discrete planes.  To a discrete plane is associated a two dimensional word. The study of these words is an interesting problem. The complexity of such a word is not known, the first step in its computation is the following article \cite{Dom.Jam.Verg.Vui.10}.  

\begin{figure}
\centering
\includegraphics{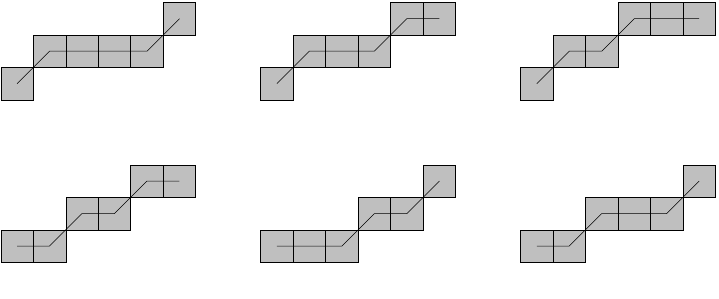}
\caption{There exist six discrete segments of length 5 and height 2. \label{fig::fig_ex_seg}}
\end{figure}

\section{Basic notions and notation}

Let $\{0,1\}^*$ and $\{0,1\}^\N$ be the set of respectively finite and
infinite words on the alphabet $\{0,1\}$. We denote the empty word by
$\epsilon$. For any word $w\in\{0,1\}^*$, $|w|$ denotes the length of
$w$, and $|w|_0$ and $|w|_1$ denote respectively the number of $0$'s
and $1$'s in $w$. $|w|_1$ is also called the \emph{height} of $w$.  
A (finite or infinite) word $w$ is \emph{balanced} if and only if for
any finite subwords $u$ and $v$ of $w$ such that $|u|=|v|$, we have
$\bigl||u|_0-|v|_0\bigr|\leq 1$. A (finite or infinite) word $w$ is of type $0$ \resp{type
  $1$} if the word $w$ does not contain $11$  \resp{the word $00$}.  
We denote by $\bw$ the set of finite balanced words and by $\bw^0$
\resp{$\bw^1$} the set of finite balanced words of type $0$
\resp{$1$}. 

Let $L,h \in \N$ and $\alpha, \beta \in \{0,1\}^*$. We denote by
$\bw_{\alpha,\beta}(L,h)$ the set of elements of $\bw$ of length
$L$ and height $h$, of which $\alpha$ is a prefix and $\beta$ is
a suffix. Note that $\alpha$ and $\beta$ may
overlap. For short, we usually write $\bw(L,h)$ instead of
$\bw_{\epsilon,\epsilon}(L,h)$. Observe that $\bw(L,h)$ is the set of
finite balanced words which encode the discrete segments between
$(0,0)$ and $(L,h)$. Remark also that $L-h$ is the width of the word, that is the number of zero's. We can count by height or by width, it is the same and this symmetry is used several times in the paper.

We extend the definition of the function $s(L,h)$ on $\Z^2$ by: 
\[
s(L,h) = \left\{\begin{array}[c]{l@{\qquad}l}
\# \bw(L, h \bmod L) & \mbox{if $L> 0$}, \\
1 & \mbox{if $L=0$ and $h=0$}, \\
0 & \mbox{if $L<0$ or $L=0$ and $h\neq0$}
\end{array}\right.
\]
Observe that for $0 \leq h \leq L$, since $\#\bw(L,L)=\#\bw(L,0)$, one
has $s(L,h)= \# \bw(L,h)$. 

For $0\leq h \leq L$ and $\alpha,\beta \in \{0,1\}^*$ we denote by
$s_{\alpha,\beta}(L,h)$ the cardinal of
$\bw_{\alpha,\beta}(L,h)$. Notice that $s_{\alpha,\beta}(L,h) =
s_{\overline\alpha,\overline\beta}(L,L-h)$, 
where $\overline{w}$ is the word obtained by replacing the $0$'s with
$1$'s and the $1$'s with $0$'s in $w$. 

\section{General case}

\subsection{Main theorem}
In the present section, we prove the following result:
\begin{theorem}\label{thm::thm1}
For all $L,h \in \N$ satisfying $0 \leq h \leq L/2$, one has:
\[
s(L,h) = s(L-h-1,h) + s(L-h,h) - s(L-2h-1,h) + s(h-1,L-2) + s(h-1,L-1).
\]
\end{theorem}

In order to prove Theorem~\ref{thm::thm1}, let us now introduce some
technical definitions and lemmas. Let $\varphi$ be the morphism
defined on  $\{0,1\}^*$ and $\{0,1\}^\N$ by:
\[
\varphi:
\begin{array}{ccc}
                  0 & \mapsto & 0 \\
                  1 & \mapsto & 01
\end{array}
\]
Let us recall that $\varphi$ is a Sturmian morphism, that is, for any Sturmian sequence $u$, the sequence $\varphi(u)$ is Sturmian \cite{Par1997,MigSeeb1993}. 
 Moreover: 
\begin{lemma}\cite{Lot2002}\label{lem::Lot2002}
Let $w \in \{0,1\}^\N$.
\begin{enumerate}
\item If $0w$ is Sturmian of type $0$, then there exists a unique Sturmian sequence $u$ satisfying $\varphi(u)=0w$. 
\item $w$ is Sturmian if and only if so is $\varphi(w)$.
\end{enumerate}
\end{lemma}
Since every balanced word is a factor of a Sturmian word, we directly deduce:
\begin{corollary}\label{cor::finite_balance}
If a finite word $w \in \{0,1\}^*$ is balanced then so is $\varphi(w)$.
\end{corollary}
 
\begin{definition}[$0$-erasing map]
Let $\theta : \{0,1\}^* \rightarrow \{0,1\}^*$ be the map defined by the recurrence relations:
\[
\begin{array}{lcll}
\theta (\epsilon)  & =  & \epsilon, \\
\theta (0^{\alpha+1}) & = & 0^\alpha & \quad \text{ for } \alpha \geq 0, \\
\theta (1v) & = & 1\theta (v), \\
\theta (0^{\alpha+1}1v) & = & 0^{\alpha}1 \theta(v) & \quad \text{ for } \alpha \geq 0, \\

\end{array}
\]
\end{definition}
Roughly speaking, $\theta$ \emph{erases} a $0$ in each \emph{maximal range} of $0$ in a given
word. In some sense, $\theta$ is the inverse of $\varphi$. Let us now prove some key
properties of $\theta$: 

\begin{lemma}\label{lem::lem2}
Consider the set $\bw^0_{0,1} =\{u \in \bw^0,\, \exists w \in \{0,1\}^*, \, u=0w1\}$ of words in $\bw^0$ of the form $0w1$
with $w \in \{0,1\}^*$. Then 
\begin{itemize}
\item The map $\theta$ restricted to $\bw^0_{0,1}$ is a bijection on $\bw_{0,1}$. The map $\varphi$ restricted to $\bw_{0,1}$ is a bijection on $\bw^0_{0,1}$.

\item Moreover we have $\theta(\varphi(w1)) = w1$ for all $w$.
\end{itemize}
\end{lemma}
\begin{proof}
By induction on $|w|_1$.
\begin{enumerate}
\item If $|w|_1=0$ then $w=0^\alpha$ for some $\alpha \geq 0$ and we have $\varphi (
  \theta(0 0^\alpha 1)) = \varphi(0^\alpha 1) = 0^{\alpha+1}1=0w1$. 
\item Assume $|w|_1\geq 1$ and the result holds for all $u$ such that $|u|_1<|w|_1$. We
  have $w=0^\alpha 1 w'$ for some $\alpha \geq 0$. By 
  assumption, the letter $1$ is 
  isolated in $0w1$, so that $w' \neq \epsilon$ and $w'$ starts with the letter
  $0$. Hence, 
\begin{equation*}
\varphi ( \theta(0 w 1)) = \varphi ( \theta(0^{\alpha+1} 1w'1)) =   \varphi (
0^{\alpha} 1 \theta(w'1)) = 0^\alpha 0 1 \varphi (  \theta(w'1)). 
\end{equation*}
By the induction hypothesis, we obtain 
\begin{equation*}
\varphi ( \theta(0 w 1)) =0^{\alpha+1}1w'1 = 0w1.
\end{equation*}
\end{enumerate}
\end{proof}
\begin{example}
We have by straightforward computations:
$\varphi(\theta(11))=0101.$ Thus the last equation of Lemma \ref{lem::lem2} is not true everywhere.
\end{example}
\begin{lemma}\label{lem::theta0bal}
Let $w \in \{0,1\}^*$. If $w$ is balanced then so is $\theta(w)$.
\end{lemma}
\begin{proof}
\begin{enumerate}
\item If $w$ is of type $1$ (i.e. the letter $0$ is isolated in $w$), then we verify that
  $\theta(w)=1^\alpha$ for some integer $\alpha$. Hence it is balanced.
\item Assume now that $w \in \bw_0$. 

\begin{itemize}
\item 
There exist $\alpha \in \{0,1\}$, $\beta \in \N$ and a Sturmian
 sequence $u$ of type $0$ such that the sequence $0^\alpha w 0^\beta 1 u$
  is Sturmian and starts with the letter $0$. Notice that $u$ starts
  with the letter $0$ too.
\item By point 1 of Lemma~\ref{lem::Lot2002}, there exists a Sturmian sequence $u'$ such that $u=\varphi(u')$. 
\item We have 
\begin{equation*}
\begin{array}{llll}
\varphi(\theta(0^\alpha w 0^\beta 1)u') & = & \varphi(\theta(0^\alpha w 0^\beta 1))\varphi(u') 
   & \quad \text{ since $\varphi$ is a morphism} \\
& = & 0^\alpha w 0^\beta 1 u & \quad \text{ by Lemma~\ref{lem::lem2}}.
\end{array}
\end{equation*}
and by point 2 of Lemma~\ref{lem::Lot2002}, $\theta(0^\alpha w 0^\beta 1)u'$ is
Sturmian because $0^\alpha w 0^\beta 1 u$ is Sturmian. Hence
$\theta(0^\alpha w 0^\beta 1)$ is balanced as a factor of a balanced word. Finally, we prove by
induction on $|w|_1$ that $\theta(0^\alpha w0^\beta1)= 0^{\alpha'} \theta(w) 0^{\beta'}1$
for some integers $\alpha'$ and $\beta'$, so that $\theta(w)$ is balanced as a factor of
a balanced word.
\end{itemize}
\end{enumerate}
\end{proof}

The last technical property of $\theta$ we need is:
\begin{lemma}\label{lem::bijections}~
\begin{enumerate}
\item\label{lem::bijections::s00}
  If $L\geq 2h+1$ then $\theta$ is a bijection from $\bw_{0,0}(L,h)$ to
  $\bw_{\epsilon,\epsilon}(L-(h+1),h).$ 
  
  \item\label{lem::bijections::s01}
  If $L\geq 2h$ then $\theta$ is a bijection from ${\bw_{0,1}(L,h)}$ to
  $\bw_{\epsilon,1}(L-h,h)$ and from
  ${\bw_{1,0}(L,h)}$ to $\bw_{1,\epsilon}(L-h,h)$.

\item\label{lem::bijections::s11}
  If $L\geq 2h-1$ then $\theta$ is a bijection from ${\bw_{1,1}(L,h)}$ to
  $\bw_{1,1}(L-(h-1),h).$ 
\end{enumerate}
\end{lemma}
\begin{proof}
If $h=0$ and $L\neq 0$, $\bw_{0,0}(L,h) = \{0^L\}$ and $\bw_{\epsilon,\epsilon}(L-(h+1),h) =
\{0^{L-1}\} = \{\theta(0^L)\}$. All others sets are empty so that the result obviously
holds.  In the rest of the proof, we assume $h\geq1$. We prove the result for
$\bw_{0,1}(L,h)$. The proof of other cases is similar and left to the reader. 

Notice first that $\bw_{0,1}(L,h) \subset \bw^0$ iff $L\geq 2h$. Indeed, if $L=2h$, then
$\bw_{0,1}(L,h)  = \{(01)^h\} \subset \bw^0$. Now, if 
$L> 2h$, by the pigeonhole principle, any $w\in\bw_{0,1}(L,h)$ must contain the
subword $00$. Since $w$ is balanced, it cannot contain the subword $11$ hence the letter 1
is isolated. Conversely, if the letter 1 is isolated in $w$, then $w$ must contain at least
$h$ 0's, hence $L\geq 2h$. 
\begin{itemize}
\item $\theta({\bw_{0,1}(L,h)}) \subset \bw_{\epsilon,1}(L-h,h)$

Let $w\in\bw_{0,1}(L,h)$. By an easy induction on $h$, we show that
$|\theta(w)| = L-h$ and $|\theta(w)|_1 = |w|_1 = h$. Furthermore, from
Lemma~\ref{lem::theta0bal}, 
$\theta(w)$ is balanced so that $\theta(w) 
\in \bw(L-h,h)$. Now from the definition of $\theta$, if $1$ is a suffix of $w$,
then it is also a suffix of $\theta(w)$ so that $\theta(w)
\in \bw_{\epsilon,1}(L-h,h)$.

\item $\theta:\bw_{0,1}(L,h)\rightarrow \bw_{\epsilon,1}(L-h,h)$ is
  injective.

Let $u,v \in \bw_{0,1}(L,h)$. We have $u = 0^{\alpha+1}1u'$, $v = 0^{\beta+1}1v'$ and 
\[
\theta(u) = \theta(v) \iff
0^{\alpha}1\theta(u') = 0^{\beta}1\theta(v') \iff \alpha = \beta \wedge \theta(u') =
\theta(v').
\]
Now, either $h=1$ and $u'=v'=\epsilon$ so that $u=v$ or $h>1$ and $u',v'\in
\bw_{0,1}(L-\alpha-1,h-1)$. We get the result by induction on $h$.

\item $\theta:\bw_{0,1}(L,h)\rightarrow \bw_{\epsilon,1}(L-h,h)$ is
  surjective.

Let $w\in \bw_{\epsilon,1}(L-h,h)$. We have $w' = \varphi(w) \in
\bw_{0,1}(L,h)$. Indeed, $w'$ is balanced because $w$ is,
$|w'| = |w|_0 + 2|w|_1 = L-2h+2h=L$ and $|w'|_1 = |w|_1 = h$ so that
$w' \in\bw(L,h)$. Since $1$ is a suffix of $w$, it is also a suffix of
$w'$ and from Lemma~\ref{lem::lem2}, $\theta(w') = w$.
\end{itemize}
 
As already said, the proof of other cases is similar. 
To prove that $\theta : \bw_{0,0}(L,h) \rightarrow \bw_{\epsilon,\epsilon}(L-(h+1),h)$
is surjective, we consider for each $w\in \bw_{\epsilon,\epsilon}(L-(h+1),h)$, $w' =
\varphi(w)0$. Then we have $w' \in\bw_{0,0}(L,h)$ and $\theta(w') = w$.  
For $\theta : \bw_{1,0}(L,h) \rightarrow \bw_{1,\epsilon}(L-h,h)$, we
consider for each $w\in \bw_{1,\epsilon}(L-h,h)$, $w' = w''0$ where $\varphi(w) = 0w''$. Finally, for $\theta : \bw_{1,1}(L,h)
\rightarrow \bw_{1,1}(L-(h-1),h)$, we consider for each $w \in \bw_{1,1}(L-(h-1),h)$, $w' = w''$ where $\varphi(w) = 0w''$.

\end{proof}

\begin{corollary}\label{cor::s11}
For all $L,h$ such that $2\leq h\leq L$, $s_{1,1}(L,h) = s_{1,1}(h+(L-h)\bmod(h-1),h)$.
\end{corollary}

\begin{proof}
Follows from case~\ref{lem::bijections::s11} by induction on $q=\left\lfloor
  \frac{L-h}{h-1}\right\rfloor$. 

\end{proof}

\begin{lemma}\label{lemme::s00Lh}
For all $L,h$ such that $0 \leq h \leq L$, $s_{0,0}(L,h) = s(L-h-1,h)$ and $s_{1,1}(L,h) = s(h-1,L-1)$.
\end{lemma}
\begin{proof}
We distinguish several cases.
\begin{itemize}
\item If $2h<L$, the result is an immediate consequence of
  case~\ref{lem::bijections::s00} of Lemma~\ref{lem::bijections}. 
\item If $h+1<L\leq 2h$ (which implies $h\geq2$), we have
\[
\begin{array}{lcll} 
s_{0,0}(L,h)
 &=& s_{1,1}(L,L-h) &\quad \mbox{by exchanging $0$'s and $1$'s}, \\
 &=& s_{1,1}(L-h+h\bmod(L-h-1),L-h) & \quad \mbox{by Corollary~\ref{cor::s11}},\\
 &=& s_{0,0}(L-h+h\bmod(L-h-1),h\bmod(L-h-1)) & \quad \mbox{by exchanging $0$'s and $1$'s}, \\
 &=& s(L-h-1,h\bmod(L-h-1)) & \quad \mbox{by case~\ref{lem::bijections::s00} of
   Lemma~\ref{lem::bijections}},\\ 
 &=& s(L-h-1,h) & \quad \mbox{by definition of $s(L,h)$}\\
\end{array}
\]

\item If $L=h+1$ we have $L-h-1=0$. Either $h=0$ and $L=1$ in which case we have
  $s_{0,0}(1,0) = \#\{0\} =   1 = \#\{\epsilon\} = s(0,0)$, or $h>0$ and we have $s_{0,0}(L,h) = 0
  = s(0,h)$. 
\item If $L=h$ we have $s_{0,0}(h,h) = 0 = s(-1,h)$.
\end{itemize}
Since $s_{1,1}(L,h)=s_{0,0}(L,L-h)$, we immediately obtain
$s_{1,1}(L,h)=s(h-1,L-h)$. Moreover $s(L,h)=s(L,h \bmod L)$, so that 
$s_{1,1}(L,h) = s(h-1,L-1)$. 

\end{proof}

We are now ready to prove the main theorem.\\

\begin{proof}[of Theorem~\ref{thm::thm1}]
The property holds for $L=0$. If $L>0$, then the following disjoint
union holds: 
$$\bw(L,h) = \bw_{0,0}(L,h) \uplus \bw_{0,1}(L,h) \uplus \bw_{1,0}(L,h) \uplus \bw_{1,1}(L,h),$$ 
and, consequently:
$$s(L,h) = s_{0,0}(L,h) + s_{0,1}(L,h) + s_{1,0}(L,h) + s_{1,1}(L,h).$$
From Lemmas~\ref{lem::bijections} and \ref{lemme::s00Lh}, it follows that
\begin{eqnarray*}
s(L,h) & = & s(L-h-1,h) + s_{\epsilon,1}(L-h,h) + s_{1,\epsilon}(L-h,h) + s(h-1,L-1) \\
       & = & s(L-h-1,h) + s(L-h,h) + s_{1,1}(L-h,h) \\ &&- s_{0,0}(L-h,h) + s(h-1,L-1) \\
       & = & s(L-h-1,h) + s(L-h,h) + s(h-1,L-2)\\&& - s(L-2h-1,h) + s(h-1,L-1).
\end{eqnarray*}
\end{proof}

\subsection{Remark}
To summarize, we can compute the formula for $s(L,h)$ for all integers $L,h \in \Z$. Indeed if $L$ is negative, then it is null. If $h=0$, then $S(L,0)=1$. If $L=0$ then $S(0,h)=1$. The other values can be computed with the statement of Theorem~\ref{thm::thm1} and the relation $s(L,h)=s(L, L-h)$ which is obtained by exchanging $0$'s and $1$'s.
Sample values of $s(L,h)$ are given in Table~\ref{tab::sLh}. The sum of elements in a row give the value of $s(L)$.
\begin{table}[h]
\begin{center}
\tabcolsep=1ex
\begin{tabular}{r@{~~}||*{11}r}
\multicolumn{1}{r||}{$L\backslash h\!$}     
    &  ~0 &  1 &  2 &  3 &  4 &  5 &  6 &  7 &  8 &  9 & 10 \\
 \hline
  0 &  1 &    &    &    &    &    &    &    &    &    &    \\
  1 &  1 &  1 &    &    &    &    &    &    &    &    &    \\
  2 &  1 &  2 &  1 &    &    &    &    &    &    &    &    \\
  3 &  1 &  3 &  3 &  1 &    &    &    &    &    &    &    \\
  4 &  1 &  4 &  4 &  4 &  1 &    &    &    &    &    &    \\
  5 &  1 &  5 &  6 &  6 &  5 &  1 &    &    &    &    &    \\
  6 &  1 &  6 &  8 &  6 &  8 &  6 &  1 &    &    &    &    \\
  7 &  1 &  7 & 11 &  8 &  8 & 11 &  7 &  1 &    &    &    \\
  8 &  1 &  8 & 13 & 12 &  8 & 12 & 13 &  8 &  1 &    &    \\
  9 &  1 &  9 & 17 & 13 & 12 & 12 & 13 & 17 &  9 &  1 &    \\
 10 &  1 & 10 & 20 & 16 & 16 & 10 & 16 & 16 & 20 & 10 &  1 \\
\end{tabular}
\end{center}
\caption{Sample values of $s(L,h)$ for $0\leq h\leq L\leq10$}
\label{tab::sLh}
\end{table}

\subsection{An explicit formula for $s(L,2)$}

Using the recurrence formula of Theorem~\ref{thm::thm1}, we can
deduce an explicit formula for some particular cases. Actually we are not able to give an explicit formula in all cases.
For instance, one has:
\begin{proposition}
Let $L\geq0$ be an integer. Then, one has:
\begin{equation*}
s(L,2) = \left\lfloor \dfrac{(L+1)^2+2}{6} \right\rfloor.
\end{equation*}
\end{proposition}
\begin{proof} By induction on $L$.
One checks that the result holds for $L \in \{0,1,2,3,4\}$. Assume $L \geq 5$ and the
result holds for all nonnegative integers smaller than $L$. From Theorem~\ref{thm::thm1}, one deduces
$s(L,2) -s(L-3,2) = s(L-2,2)-s(L-5,2)+2$. 

For all $L\geq3$, let $u_L = s(L,2)-s(L-3,2)$. Then, $u_{L+2}=u_{L}+2$
and one obtains $u_L=L-1+(L \bmod 2)$. By induction, it follows: 
\begin{equation*}
s(L,2) = \left\lfloor \dfrac{(L-2)^2+2}{6} \right\rfloor + L-1 + (L \bmod 2)
 = \left\lfloor \dfrac{L^2+2L}{6} \right\rfloor + (L \bmod 2).
\end{equation*}
Finally, it suffices to check that:
\begin{equation*}
\left\lfloor \dfrac{L^2+2L}{6}+\dfrac{1}{2} \right\rfloor 
= \left\lfloor \dfrac{L^2+2L}{6} \right\rfloor + (L \bmod 2).
\end{equation*}
By considering the remainder of $L$ modulo $6$, we obtain that the
fractional part of $\frac{L^2+2L}{6}$ is strictly less than
$\frac{1}{2}$ if and only if $L$ is even. The result follows. 
\end{proof}

\subsection{Generating functions of $s(L,h)$}

A classical way to obtain an explicit formula of a given function consists in computing
its generating function. In this section, we exhibit for each $h\geq0$ the generating
function $\mathcal{S}_h(X)$ of $s(L,h)$, namely $\smash{\dsp \mathcal{S}_h(X) = \sum_{L\geq 0} s(L,h)X^L}$. Let us recall that, in that case, $s(L,h)=\dfrac{\mathcal{S}^{(L)}_h(0)}{L!}$, where $\mathcal{S}^{(L)}_h$ is the derivative \\

of order $L$ of $\mathcal{S}^{(L)}$.  

\begin{theorem}
One has: $\dsp\mathcal{S}_0(X) = \frac{1}{1-X}$, $\dsp\mathcal{S}_1(X) = \frac{X}{1-X^2}$
and for all $h\geq 2$,
$$\mathcal{S}_h(X)=\frac{\mathcal{F}_h(X)}{(1-X^{h-1})(1-X^h)(1-X^{h+1})},$$
where
$$\mathcal{F}_h(X)=(1-X^{h-1})(V_{2h,h}X^h-V_{h-1,h}X^{h+1}-X^{2h-1})+(1+X)B_h,$$
$$V_{n,h}=\displaystyle\sum_{L=0}^{n-1}s(L,h)X^L,$$
$$B_h=\displaystyle\sum_{r=0}^{h-2}s(h-1,r)X^{r+2h-1}.$$
\end{theorem}
\begin{proof}
We have immediately the two first formulas.
From the previous recurrence, for $h\geq2$, we get:
\begin{eqnarray*}
\mathcal{S}_h(X) 
&=& \sum_{L=0}^{2h-1}s(L,h)X^L + \sum_{L\geq 2h}\begin{array}[t]{@{}l@{}}\bigl(s(L-h-1,h)+s(L-h,h)-s(L-2h-1,h)\\
                                                                      \qquad+s(h-1,L-1)+s(h-1,L-2)\bigr)X^L
                                             \end{array}\\
&&\mbox{by using the recurrence of Th.~\ref{thm::thm1}}.\\
&=& \sum_{L=0}^{2h-1}s(L,h)X^L + X^{h+1}\sum_{L\geq h-1}s(L,h)X^L+X^h\sum_{L\geq  h}s(L,h)X^L\\ &&\qquad
  -X^{2h+1}\sum_{L\geq -1}s(L,h)X^L+\sum_{L\geq 2h}\bigl(s(h-1,L-1)+s(h-1,L-2)\bigr)X^L\\
&=& \sum_{L=0}^{2h-1}s(L,h)X^L + X^{h+1}\left(\mathcal{S}_h(X)-\sum_{L=0}^{h-2}s(L,h)X^L\right) +\\&& X^h\left(\mathcal{S}_h(X)-\sum_{L=0}^{h-1}s(L,h)X^L\right)\\ &&\qquad 
  - X^{2h+1} \mathcal{S}_h(X) + (X+X^2)\sum_{L\geq 2h-2}s(h-1,L)X^{L} - s(h-1,0)X^{2h-1}\\
&=& (X^h + X^{h+1} - X^{2h+1}) \mathcal{S}_h(X)+\sum_{L=0}^{2h-1}s(L,h)X^L -X^{h+1}\sum_{L=0}^{h-2}s(L,h)X^L\\ &&\qquad  -X^h\sum_{L=0}^{h-1}s(L,h)X^L
  + (X+X^2)\sum_{q\geq2}\sum_{r=0}^{h-2}s(h-1,r)X^{q\,(h-1)+r} - X^{2h-1}\\
&&\mbox{by setting $L = q\times (h-1)+r$}.\\
&=& (X^h + X^{h+1} - X^{2h+1}) \mathcal{S}_h(X)+\sum_{L=0}^{2h-1}s(L,h)X^L - X^{h+1}\sum_{L=0}^{h-2}s(L,h)X^L - \\ &&\qquad X^{h}\sum_{L=0}^{h-1}s(L,h)X^L
  + (1+X) \dfrac{X^{2h-1}}{1-X^{h-1}} \sum_{r=0}^{h-2}s(h-1,r)X^{r} - X^{2h-1}.\\
\end{eqnarray*}
Finally, we get the formula
\begin{eqnarray*}
  \mathcal{S}_h(X) 
  &=& \dfrac{\mathcal{F}_h(X)}{(1-X^{h-1})(1-X^{h})(1 - X^{h+1})}
\end{eqnarray*}
where $\mathcal{F}_h\in\Z[X]$ and $\deg(\mathcal{F}_h) \leq 3h-2$.
\end{proof}

Notice that the previous equality does provide a closed formula for $\mathcal{S}_h(X)$
although it still depends on $s(L,h)$ because each sum is finite. Sample values of
$\mathcal{S}_h(X)$ are given in Table~\ref{tab::sgen-func}.

\begin{table}[h]
\begin{eqnarray*}
\mathcal{S}_2(X) &=& \dfrac{X+X^3}{(1-X)(1-X^2)(1-X^3)}\\
\mathcal{S}_3(X) &=& \dfrac{X+2X^2+X^4+2X^5}{(1-X^2)(1-X^3)(1-X^4)}\\
\mathcal{S}_4(X) &=& \dfrac{X+X^2+3X^3+3X^5+3X^6+3X^7}{(1-X^3)(1-X^4)(1-X^5)}\\
\mathcal{S}_5(X) &=& \dfrac{X+2X^2+3X^3+4X^4+3X^6+5X^7+3X^8+4X^9+X^{12}}{(1-X^4)(1-X^5)(1-X^6)}\\
\mathcal{S}_6(X) &=& \dfrac{X+X^2+X^3+4X^4+5X^5+5X^7+10X^8+7X^9+6X^{10}+5X^{11}+X^{14}}{(1-X^5)(1-X^6)(1-X^7)}\\
\end{eqnarray*}
\caption{Sample values of $\mathcal{S}_h(X)$}
\label{tab::sgen-func}
\end{table}

\subsection{Asymptotic behaviour of $s(L,h)$}

Using the generating functions we just computed, we may deduce an expression of $s(L,h)$
which highlights its asymptotic behaviour when $L$ grows.

We prove the following theorem:
\begin{theorem}\label{th::sgen-quadratic}
For all $h\geq2$, there exist $u_0,\ldots,u_{h-2},v_0,\ldots,v_{h-1},w_0,\ldots,w_{h} \in
\Q$ such that
\[
  \forall L\geq 0,\quad s(L,h)= \alpha\,L^2 + \beta\,L + u_{L\bmod(h-1)}+v_{L\bmod h}+w_{L\bmod(h+1)}
\]
with
$\dsp\alpha = \dfrac{1}{h\,(h^2-1)}\sum_{i=1}^{h-1}(h-i)\varphi(i)$
and
$\dsp\beta = \dfrac{1}{h\,(h+1)}\sum_{i=1}^{h}\varphi(i)$ where $\varphi$ is Euler's
totient function.
\end{theorem}
Before proving this theorem, we need some preliminary results.

\begin{lemma}\label{lem::sgen-partfrac}
For all $h\geq2$, there exist $R,A,B,C\in\Q[X]$ such that $\deg(R)<3$, $\deg(A)<h-1$,
$\deg(B)<h$, $\deg(C)<h+1$ and
\begin{equation}
\mathcal{S}_h(X) = \dfrac{R(X)}{(1-X)^3}+\dfrac{A(X)}{1-X^{h-1}}+\dfrac{B(X)}{1-X^h}+\dfrac{C(X)}{1-X^{h+1}}.
\label{eqn::sgen-partfrac}
\end{equation}
\end{lemma}
\begin{proof}
We have
\[
  \mathcal{S}_h(X) = \dfrac{\mathcal{F}_h(X)}{(1-X^{h-1})(1-X^{h})(1 - X^{h+1})} 
  = \dfrac{\mathcal{F}_h(X)}{(1-X)^3\,\dfrac{1-X^{h-1}}{1-X}\,\dfrac{1-X^{h}}{1-X}\,\dfrac{1 - X^{h+1}}{1-X}}
\]
where $\mathcal{F}_h\in\Z[X]$ and $\deg(\mathcal{F}_h) \leq 3h-2$. 

If $h$ is even, $(1-X)^3$, $\frac{1-X^{h-1}}{1-X}$, $\frac{1-X^{h}}{1-X}$ and 
$\frac{1 - X^{h+1}}{1-X}$ are pairwise coprime so that  
\[
  \mathcal{S}_h = \dfrac{R(X)}{(1-X)^3} + \dfrac{A_0(X)}{\,\dfrac{1-X^{h-1}}{1-X}\,} +
  \dfrac{B_0(X)}{\,\frac{1-X^h}{1-X}\,} + \dfrac{C_0(X)}{\,\dfrac{1-X^{h+1}}{1-X}\,}
\]
for some $R,A_0,B_0,C_0\in \Q[X]$ with $\deg(R)<3$, $\deg(A_0)<h-2$,
$\deg(B_0)<h-1$ and $\deg(C_0)<h$. The result follows with $A(X)=(1-X)A_0(X)$,
$B(X)=(1-X)B_0(X)$ and $C(X)=(1-X)C_0(X)$.

If $h$ is odd, $\frac{1-X^{h-1}}{1-X}$ and $\frac{1-X^{h+1}}{1-X}$ are divisible by
$1+X$. But in this case, we notice that $\mathcal{F}_h(-1)=0$ so that $\mathcal{F}_h(X)$ is also divisible by
$1+X$ and we may write
\[
 \mathcal{S}_h(X)
 = \dfrac{\dfrac{\mathcal{F}_h(X)}{1+X}}{(1-X)^3\,\dfrac{1-X^{h-1}}{1-X}\,\dfrac{1-X^{h}}{1-X}\,\dfrac{1-X^{h+1}}{1-X^2}}
\]
where $(1-X)^3$, $\frac{1-X^{h-1}}{1-X}$, $\frac{1-X^{h}}{1-X}$ and
$\frac{1-X^{h+1}}{1-X^2}$ are pairwise coprime. Thus we get
\[
 \mathcal{S}_h = \dfrac{R(X)}{(1-X)^3} + \dfrac{A_0(X)}{\,\dfrac{1-X^{h-1}}{1-X}\,} +
 \dfrac{B_0(X)}{\,\dfrac{1-X^h}{1-X}\,} + \dfrac{C_0(X)}{\,\dfrac{1-X^{h+1}}{1-X^2}\,}
\]
for some $R,A_0,B_0,C_0\in \Q[X]$ with $\deg(R)<3$, $\deg(A_0)<h-2$,
$\deg(B_0)<h-1$ and $\deg(C_0)<h-1$. The result follows with $A(X)=(1-X)A_0(X)$,
$B(X)=(1-X)B_0(X)$ and $C(X)=(1-X^2)C_0(X)$.

\end{proof}

\begin{lemma}\label{lem::rshr}
We recall that $s(L)$ is the number of balanced words of length $L$, 
and that it is equal to $\displaystyle\sum_{h}s(L,h)$.
For all $L\geq0$ we have,
\[
  \sum_{h=0}^{L-1}s(L,h)\,h = \dfrac{L}{2}(s(L)-2).
\]
\end{lemma}
\begin{proof}
Let $\dsp Z_L = \sum_{h=0}^{L}s(L,h)\,h$. Then

\[
  Z_L= \sum_{u=0}^{L}s(L,L-u)(L-u)  \quad \mbox{by setting $h= L-u$}\]
  
  $$Z_L= \sum_{u=0}^{L}s(L,u)(L-u) \mbox{because $s(L,L-u) = s(L,u)$ for all $L,u \in\Z$}$$
 $$ Z_L= \vphantom{\sum_{u=0}^{L}}L\,s(L)-Z_L$$

Hence $ Z_L = \frac{L}{2}s(L)$ and
$ \sum_{h=0}^{L-1}s(L,h)\,h = Z_L - L = \frac{L}{2}(s(L)-2).$

\end{proof}

\begin{lemma}\label{lem::diff-sLh}
For all $h\geq1$,
\[
  \sum_{L=h}^{2h-1}s(L,h) - \sum_{L=0}^{h-1}s(L,h) = s(h)+s(h-1)-(h+1)
\]
\end{lemma}
\begin{proof}
\begin{eqnarray*}
  \sum_{L=h}^{2h-1}s(L,h) - \sum_{L=0}^{h-1}s(L,h)\span\span \\
  &=& \sum_{L=0}^{h-1}s(L+h,h) - \sum_{L=0}^{h-1}s(L,h) \\
   &=& \sum_{L=0}^{h-1}(s(h,L)+s(h-1,L)-s(h-L-1,L)+s(L-1,L+h-1) \\
&&+s(L-1,L+h-2)) - \sum_{L=0}^{h-1}s(L,h)\\ &&\mbox{by using the recurrence relation from Th.~\ref{thm::thm1}}\\
  &=& \sum_{L=0}^{h-1}(s(h,L)+s(h-1,L)-s(h-1-L,h-1)\\&&+s(L-1,h)+s(L-1,h-1)) - \sum_{L=0}^{h-1}s(L,h) \\
&&\mbox{by using the relation $s(L,h) = s(L,h\bmod L)$ on 3rd, 4th and 5th terms}\\
 &=& s(h)-1+s(h-1)-\left(\sum_{u=0}^{h-1}s(u,h-1)-\sum_{L=0}^{h-1}s(L-1,h-1)\right)
  \\
&& - \left(\sum_{L=0}^{h-1}s(L,h)-\sum_{L=0}^{h-1}s(L-1,h)\right) \mbox{by setting $u = h-1-L$}\\
  &=& s(h)+s(h-1)- (h+1). \\
\end{eqnarray*}

\end{proof}

We are now ready to prove the main theorem of this section.\\

\begin{proof}[Proof of Theorem~\ref{th::sgen-quadratic}]
We first prove existence. In Equation~\ref{eqn::sgen-partfrac}, we write

  $$  R(X) = r_0+r_1\,(1-X)+r_2\,(1-X)^2.$$
  $$  A(X) = \sum_{k=0}^{h-2} a_k X^k, \quad
    B(X) = \sum_{k=0}^{h-1} b_k X^k, \quad
    C(X) = \sum_{k=0}^{h} c_k X^k.
  $$
Thus
\begin{eqnarray*}
\mathcal{S}_h(X) &=& \dfrac{r_0}{(1-X)^3}+\dfrac{r_1}{(1-X)^2}+\dfrac{r_2}{1-X}
+\dfrac{A(X)}{1-X^{h-1}}+\dfrac{B(X)}{1-X^h}+\dfrac{C(X)}{1-X^{h+1}}.
\end{eqnarray*}
The Taylor expansions of $\dfrac{1}{(1-X)^3},\dfrac{1}{(1-X)^2},\dfrac{1}{1-X}$ give
the series expansion of $\mathcal{S}_h(X)$ :
\begin{eqnarray*}
  \mathcal{S}_h(X) 
  &=& \sum_{L\geq0}r_0\,\dfrac{(L+1)(L+2)}{2}\,X^L+\sum_{L\geq0}r_1\,(L+1)\,X^L+\sum_{L\geq0}r_2\,X^L \\
  &&\qquad\qquad + \sum_{n\geq0} A(X)\,X^{n\,(h-1)}  + \sum_{n\geq0} B(X)\,X^{n\,h}  + \sum_{n\geq0} C(X)X^{n\,(h+1)}\\
  &=& \sum_{L\geq0}\left(r_0\,\dfrac{(L+1)(L+2)}{2}+r_1\,(L+1)+r_2\right)X^L \\
  &&\qquad\qquad
     + \sum_{n\geq0} \sum_{k=0}^{h-2} a_k X^{n\,(h-1)+k}
     + \sum_{n\geq0} \sum_{k=0}^{h-1} b_k X^{n\,h+k}
     + \sum_{n\geq0} \sum_{k=0}^{h} c_k X^{n\,(h+1)+k}\\
  &=& \sum_{L\geq0}\left(\dfrac{r_0}{2}\,L^2 + \left(\dfrac{3}{2}\,r_0+r_1\right)\,L + r_0 + r_1 + r_2 + a_{L\bmod(h-1)}+b_{L\bmod
      h}+c_{L\bmod (h+1)}\right)X^L. 
\end{eqnarray*}
We get the result with $\alpha = \frac{1}{2}\,r_0$, $\beta=\frac{3}{2}\,r_0+r_1$, $u_i = a_i+r_0+r_1+r_2$, $v_i = b_i$
and $w_i = c_i$.

From the Taylor series of $(1-X)^3\,\mathcal{S}_h(X)$ at $X=1$, we get
$r_0=\frac{\mathcal{F}_h(1)}{h\,(h^2-1)}$ and $r_1 = \frac{\frac{3}{2}(h-1)\,\mathcal{F}_h(1)-\mathcal{F}_h'(1)}{h\,(h^2-1)}$. 
We have
$\mathcal{F}_h(1) = 2 \sum_{r=0}^{h-2}s(h-1,r) = 2(s(h-1)-1),$
and by Lemmas~\ref{lem::rshr} and~\ref{lem::diff-sLh},
\begin{eqnarray*}
  \mathcal{F}_h'(1) &=& -(h-1)\left(\sum_{L=h}^{2\,h-1}s(L,h) - \sum _{L=0}^{h-1}s(L,h)
    + (h-2)\right) + 2\,\sum _{r=0}^{h-2}s (h-1,r) r \\&&+ (4\,h-1) \sum _{r=0}^{h-2}s(h-1,r)\\
  &=& -(h-1)(s(h)+s(h-1)-(h+1) + (h-2)) \\ &&+2\,\dfrac{h-1}{2}(s(h-1)-2)+(4h-1)(s(h-1)-1)\\
  &=& -(h-1)(s(h)-1) +(4h-1)(s(h-1)-1).\\
\end{eqnarray*}
Thus we get finally 
\begin{eqnarray*}
  \alpha &=& \dfrac{\mathcal{F}_h(1)}{2\,h\,(h^2-1)} = \dfrac{s(h-1)-1}{h\,(h^2-1)} = \dfrac{1}{h\,(h^2-1)}\sum_{i=1}^{h-1}{(h-i)\varphi(i)}\\
  \beta &=& \dfrac{\dfrac{3}{2}\,h\,\mathcal{F}_h(1)-\mathcal{F}_h'(1)}{h\,(h^2-1)} = \dfrac{s(h)-s(h-1)}{h\,(h+1)} =\dfrac{1}{h\,(h+1)} \sum_{i=1}^{h}{\varphi(i)}.\\
\end{eqnarray*}
\end{proof}

\section{Balanced palindromes and symmetrical discrete segments}

In the present section, we focus on the case of segments which are symmetrical with
respect to the point $(L/2,h/2)$. These segments are encoded by
balanced palindromes. 

\begin{figure}
  \begin{center}
    \includegraphics{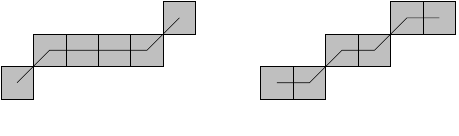}
  \end{center}
\caption{The two symmetrical segments of length 5 and height 2, and their respective
  encodings as balanced palindromes}
\end{figure}

\subsection{Recurrence formula}

This investigation is close to the general case by noticing that if $w$ is a palindrome,
then so is $\theta(w)$. We first need an additional property of the mapping $\theta$.

\begin{lemma}\label{lem::alttheta0}
For all $w \in \{0,1\}^*$ and all $\alpha \geq 0$, 
\begin{eqnarray*}
\theta(w1) &=& \theta(w)1\\
\theta(w10^{\alpha+1}) &=& \theta(w)10^\alpha
\end{eqnarray*}
\end{lemma}
\begin{proof}
Easy induction on $|w|_1$.
\end{proof}

\begin{corollary}\label{cor::theta0tilde}
For all $w \in \{0,1\}^*$, $\theta(\Tilde{w}) = \Tilde{(\theta(w))}$
\end{corollary}
\begin{proof}
By induction on $|w|_1$. If $|w|_1=0$ then $w = 0^\alpha$ for
some $\alpha \geq 0$ and the result obviously holds. Now assume that $|w|_1\geq 1$ and the
result holds for all $v$ such that $|v|_1<|w|_1$. Then
either $w=1v$ and
\[
\begin{array}{cccc}
\theta(\Tilde{w}) &=& \theta(\Tilde{(1v)}) \\
&=&\theta(\Tilde{v}1) &\\
&=& \theta(\Tilde{v})1 & \mbox{by Lemma~\ref{lem::alttheta0}}\\
&=& \Tilde{\theta(v)}1 & \mbox{by the induction hypothesis}\\
&=& \Tilde{(1\theta(v))} &\\
&=& \Tilde{\theta(1v)} & \mbox{by definition of $\theta$}\\
&=& \Tilde{\theta(w)}
\end{array}
\]
or $w=0^{\alpha+1}1v$ and
\[
\begin{array}{cccc}
\theta(\Tilde{w}) &=& \theta(\Tilde{(0^{\alpha+1}1v)}) \\
&=&\theta(\Tilde{v}10^{\alpha+1}) &\\
&=& \theta(\Tilde{v})10^\alpha & \mbox{by Lemma~\ref{lem::alttheta0}}\\
&=& \Tilde{\theta(v)}10^\alpha & \mbox{by the induction hypothesis}\\
&=& \Tilde{(0^\alpha1\theta(v))} &\\
&=& \Tilde{\theta(0^{\alpha+1}1v)} & \mbox{by definition of $\theta$}\\
&=& \Tilde{\theta(w)}.
\end{array}
\]
\end{proof}

\begin{corollary}\label{cor::theta0pal}
If $w \in\{0,1\}^*$ is a palindrome, then $\theta(w)$ is also a palindrome.
\end{corollary}
\begin{proof}
Immediate consequence of Corollary~\ref{cor::theta0tilde}
\end{proof}

We denote by $\pal(L,h)$ the set of balanced palindromes of length $L$ and height $h$,
and, for $x\in\{0,1\}$, by $\pal_x(L,h)$ the set of balanced palindromes of length $L$ and
height $h$ the first (and last) letter of which is $x$. 
We define the function $p(L,h)$ on $\Z^2$ by:
\[
p(L,h) = \left\{\begin{array}[c]{l@{\qquad}l}
\# \pal(L, h \bmod L) & \mbox{if $L> 0$} \\
1 & \mbox{if $L=0$ and $h=0$} \\
0 & \mbox{if $L<0$ or $L=0$ and $h\neq0$}
\end{array}\right.
\]
and for $0\leq h \leq L$ and $x\in\{0,1\}$, we define $p_x(L,h) = \#\pal_x(L,h)$.
We have the following properties. 

\begin{lemma}\label{lemme:pal:theta}
Let $L,h \in \N$ such that $0\leq h \leq L$. 
\begin{enumerate}
\item\label{lemme:pal:theta:00} If $L \geq 2h+1$, then $\theta$ is a bijection from $\pal_0(L,h)$
  to $\pal(L-(h+1),h)$.
\item\label{lemme:pal:theta:11} If $L \geq 2h-1$, then $\theta$ is a bijection from
  $\pal_1(L,h)$ to $\pal_1(L-(h-1),h)$.
\end{enumerate}
\end{lemma}
\begin{proof}~

\begin{enumerate}
\item Since $\pal_0(L,h) \subset \bw_{0,0}(L,h)$, from Lemma~\ref{lem::bijections}, we
  already know that $\theta(\pal_0(L,h)) \subset \bw(L-(h+1),h)$ and
  from Corollary~\ref{cor::theta0pal}, $\theta(\pal_0(L,h)) \subset
  \pal(L-(h+1),h)$. Since $\theta$ is injective on $\bw_{0,0}(L,h)$, it is also
  injective on $\pal_0(L,h)$. We are left to prove that it is surjective.
  Let $w \in \pal(L-(h+1),h)$ and $w' = \varphi(w)0$. From Lemma~\ref{lem::bijections},
  we have $w' \in \bw_{0,0}(L,h)$ and $\theta(w') = w$. We prove by induction on $|w|_1$
  that $w'$ is a  palindrome. If $|w|_1 = 0$ then $w = 0^\alpha$ for some $\alpha\geq 0$ and
  $w' = 0^{\alpha+1}$ which is trivially a palindrome. If $|w|_1 = 1$ then $w =
  0^\alpha10^\alpha$ for some $\alpha\geq 0$ and $w' = 0^{\alpha+1}10^{\alpha+1}$ which is again trivially a
  palindrome. If $|w|_1 \geq 2$, assume that $\varphi(u)0$ is a palindrome
  for all $u$ such that $|u|_1 < |w|_1$. We have $w = 0^\alpha1v10^\alpha$ for some $\alpha \geq 0$ and
  some palindrome $v$ with $|v|_1 < |w|_1$. Then $w' =
  0^{\alpha+1}1\varphi(v)010^{\alpha+1}$ is a palindrome because $\varphi(v)0$
  is a palindrome by the induction hypothesis..
\item The proof is similar. We get in the same way that $\theta(\pal_1(L,h)) \subset
  \pal_1(L-(h-1),h)$ and $\theta$ is injective on $\pal_1(L,h)$. To prove that it is
  surjective, for each $w\in \pal_1(L-(h-1),h)$ we consider $w'$ such that $\varphi(w) =
  0w'$. From Lemma~\ref{lem::bijections} we have $\theta(w') = w$ and
  $w'\in\bw_{1,1}(L,h)$. We prove like
  above that $w'$ is a palindrome so that $w' \in \pal_1(L,h)$.
\end{enumerate}

\end{proof}

\begin{lemma}\label{lem::p0Lh}
For all $L,h\in \N$ such that $0\leq h\leq L$, $p_0(L,h) = p(L-h-1,h)$ and $p_1(L,h) = p(h-1,L-1)$
\end{lemma}
\begin{proof}
Similar to the proof of Lemma~\ref{lemme::s00Lh}.
\end{proof}

From Lemma~\ref{lem::p0Lh} and the definition of $p(L,h)$, we deduce the following
recurrence for $p(L,h)$. 
\begin{theorem}
Let $L,h \in \Z$,
\[
p(L,h) = \left\{
\begin{array}{ll}
0 & \mbox{if $L<0$ or ($L=0$ and $h\neq0$)}\\
1 & \mbox{if $L\geq0$ and ($h=0$ or $h=L$)}\\
p(L,h\bmod L) & \mbox{if $L>0$ and ($h<0$ or $h>L$)}\\
p(L-h-1, h) + p(h - 1, L-1) &\mbox{otherwise}
\end{array}
\right.
\]
\end{theorem}
Sample values of $p(L,h)$ are given in Table~\ref{tab::pLh}.
\begin{table}[h]
\begin{center}
\tabcolsep=1ex
\begin{tabular}{r@{~~}||*{11}r}
$L\backslash h\!$
    &     0 &   1 &   2  &  3 &   4 &   5 &   6  &  7  &  8 &   9 &  \llap10 \\
\hline
  0 &    1  & \\ 
  1 &    1  &  1  \\  
  2 &    1  &  0 &   1 \\  
  3 &    1  &  1 &   1 &   1 \\  
  4 &    1  &  0 &   2 &   0 &   1 \\  
  5 &    1  &  1 &   2 &   2 &   1 &   1 \\  
  6 &    1  &  0 &   2 &   0 &   2 &   0 &   1 \\  
  7 &    1  &  1 &   3 &   2 &   2 &   3 &   1 &   1 \\ 
  8 &    1  &  0 &   3 &   0 &   2 &   0 &   3 &   0 &   1 \\ 
  9 &    1  &  1 &   3 &   3 &   2 &   2 &   3 &   3 &   1 &   1 \\ 
 10 &    1  &  0 &   4 &   0 &   2 &   0 &   2 &   0 &   4 &   0 &   1
\end{tabular}
\end{center}
\caption{Sample values of $p(L,h)$ for $0\leq h\leq L\leq10$}
\label{tab::pLh}
\end{table}

\subsection{Generating functions of $p(L,h)$}

In the same way we obtained generating functions for $s(L,h)$, we deduce generating
functions for $p(L,h)$ from the recurrence above. We consider the generating functions
$\smash{\dsp \mathcal{P}_h(X) = \sum_{L\geq 0} p(L,h)X^L}$.
\begin{theorem}
One has: $\dsp\mathcal{P}_0(X) = \frac{1}{1-X}$, $\dsp\mathcal{P}_1(X) = \frac{X}{1-X^2}$
and for all $h\geq 2$,
\[
\mathcal{P}_h(X) = \dfrac{1}{1-X^{h+1}}\left(\sum_{L=0}^{h-1} p(L,h)X^L
 + \dfrac{X^h}{1-X^{h-1}}\sum_{r=0}^{h-2}p(h-1,r)X^{r}\right).
\]
\end{theorem}
\begin{proof}
One has:
\begin{eqnarray*}
\mathcal{P}_0(X) &=&\sum_{L\geq0} p(L,0)X^L = \sum_{L\geq0} X^L = \frac{1}{1-X}\\
\mathcal{P}_1(X) &=&\sum_{L\geq0} p(L,1)X^L \\
&=& p(0,1) + p(1,1)X+\sum_{L\geq2} p(L,1)X^L\\
&=&X + \sum_{L\geq2} (p(L-2,1)+p(0,L-1))X^L\\
&=&X + X^2\sum_{L\geq0} p(L,1)X^L\\
&=&X + X^2\,\mathcal{P}_1(X).
\end{eqnarray*}
Hence, $\dsp{\mathcal{P}_1(X) = \frac{X}{1-X^2}}$.

For all $h\geq 2$, we have
\begin{eqnarray*}
\mathcal{P}_h(X) &=& \sum_{L\geq 0} p(L,h)X^L \\
&=& \sum_{L=0}^{h-1} p(L,h)X^L + \sum_{L\geq h} p(L,h)X^L\\
&=& \sum_{L=0}^{h-1} p(L,h)X^L + X^h\sum_{L\geq 0} p(L+h,h)X^L\\
&=& \sum_{L=0}^{h-1} p(L,h)X^L + X^h\sum_{L\geq 0} (p(L-1,h)+p(h-1,L+h-1))X^L\\
&=& \sum_{L=0}^{h-1} p(L,h)X^L + X^h\sum_{L\geq 0} p(L-1,h)X^L+ X^h\sum_{L\geq 0}p(h-1,L)X^L\\
&=& \sum_{L=0}^{h-1} p(L,h)X^L + X^{h+1}\,\mathcal{P}_h(X) + X^h\sum_{q\geq0}\sum_{r=0}^{h-2}p(h-1,q(h-1)+r)X^{q(h-1)+r}\\
&=& \sum_{L=0}^{h-1} p(L,h)X^L + X^{h+1}\,\mathcal{P}_h(X) + X^h\sum_{q\geq0}X^{q(h-1)}\sum_{r=0}^{h-2}p(h-1,r)X^{r}\\
&=& X^{h+1}\,\mathcal{P}_h(X) + \sum_{L=0}^{h-1} p(L,h)X^L + \dfrac{X^h}{1-X^{h-1}}\sum_{r=0}^{h-2}p(h-1,r)X^{r}.\\
\end{eqnarray*}
Finally, we get
\begin{eqnarray*}
\mathcal{P}_h(X) &=& \dfrac{1}{1-X^{h+1}}\left(\sum_{L=0}^{h-1} p(L,h)X^L + \dfrac{X^h}{1-X^{h-1}}\sum_{r=0}^{h-2}p(h-1,r)X^{r}\right)\\
&=& \dfrac{(1-X^{h-1})\sum_{L=0}^{h-1} p(L,h)X^L+X^h\,\sum_{r=0}^{h-2}p(h-1,r)X^{r}}{(1-X^{h-1})\,(1-X^{h+1})}\\
&=& \dfrac{\mathcal{G}_h(X)}{(1-X^{h-1})\,(1-X^{h+1})}
\end{eqnarray*}
where $\mathcal{G}_h(X)\in\Z[X]$ and $\deg(\mathcal{G}_h) \leq 2h-2$.
\end{proof}

Sample values of $\mathcal{P}_h(X)$ are given in Table~\ref{tab::palgen} 
\begin{table}[h]
\[
\begin{array}{cccccc}
\mathcal{P}_2(X) &=& \dfrac{X}{(1-X)(1-X^3)} & \mathcal{P}_3(X) &=& \dfrac{X}{(1-X^2)(1-X^4)}\\
\noalign{\vskip\jot}
\mathcal{P}_4(X) &=& \dfrac{X+X^2+X^3}{(1-X^3)(1-X^5)} & \mathcal{P}_5(X) &=& \dfrac{X+X^3+X^7}{(1-X^4)(1-X^6)}\\
\noalign{\vskip\jot}
\mathcal{P}_6(X) &=& \dfrac{X+X^2+X^3+2\,X^4+X^5+X^8}{(1-X^5)(1-X^7)}\span\span\span
\end{array}
\]
\caption{Sample values of $\mathcal{P}_h(X)$}
\label{tab::palgen}
\end{table}

\subsection{Asymptotic behaviour of $p(L,h)$}

As before, from the generating function, we deduce an expression of $p(L,h)$ which
highlights its asymptotic behaviour. We prove the following theorem.
\begin{theorem}\label{th::pLh-linear}
For all $h\geq2$ there exist $u_0,\dots,u_{h-2},v_0,\ldots,v_h \in \Q$ such that:
\begin{itemize}
\item if $h$ is even then
\[
\forall L\geq 0,\quad p(L,h) = \alpha\,L +  u_{L\bmod(h-1)} + v_{L\bmod(h+1)}
\]
\item if $h$ is odd then
\[
\forall L\geq 0,\quad p(L,h) = \alpha\,(1-(-1)^L)\,L +  u_{L\bmod(h-1)} + v_{L\bmod(h+1)}
\]
\end{itemize}
where $\dsp\alpha=\frac{1}{h^2-1}\sum_{i=1}^{\left\lceil \frac{h-1}{2} \right\rceil}\varphi(h+1-2i).$
\end{theorem}
Before proving this theorem, we need some lemmas.
\begin{lemma}\label{lem::pgen-partfrac}
For all $h\geq 2$:
\begin{itemize}
\item if $h$ is even then there exist $R,A,B\in\Q[X]$ such that $\deg(R)<2$,
$\deg(A)<h-1$, $\deg(B)<h+1$ and
\begin{equation}\label{eqn::pgen-partfrac-even}
\mathcal{P}_h(X) = \dfrac{R(X)}{(1-X)^2}+\dfrac{A(X)}{1-X^{h-1}}+\dfrac{B(X)}{1-X^{h+1}}.
\end{equation}
\item if $h$ is odd then there exist $Q,R,A,B\in\Q[X]$ such that $\deg(Q)<2$, $\deg(R)<2$,
$\deg(A)<h-1$, $\deg(B)<h+1$ and
\begin{equation}\label{eqn::pgen-partfrac-odd}
\mathcal{P}_h(X) = \dfrac{Q(X)}{(1+X)^2}+\dfrac{R(X)}{(1-X)^2}+\dfrac{A(X)}{1-X^{h-1}}+\dfrac{B(X)}{1-X^{h+1}}
\end{equation}
\end{itemize}
\end{lemma}
\begin{proof}
If $h$ is even then $\mathcal{P}_h(X)$ may be written as
\[
\mathcal{P}_h(X) = \dfrac{\mathcal{G}_h(X)}{(1-X)^2\,\dfrac{1-X^{h-1}}{1-X}\,\dfrac{1-X^{h+1}}{1-X}}.
\]
Since $(1-X)^2$, $\dfrac{1-X^{h-1}}{1-X}$ and $\dfrac{1-X^{h+1}}{1-X}$ are pairwise coprime,
there exist $R,A_0,B_0\in\Q[X]$ such that $\deg(R)<2$, $\deg(A_0)<h-2$,
$\deg(B_0)<h$ and
\[
\mathcal{P}_h(X) = \dfrac{R(X)}{(1-X)^2}+\dfrac{A_0(X)}{\dfrac{1-X^{h-1}}{1-X}}+\dfrac{B_0(X)}{\dfrac{1-X^{h+1}}{1-X}}.
\]
We get the result with $A(X)=(1-X)A_0(X)$ and $B(X)=(1-X)B_0(X)$.

If $h$ is odd then $1-X^{h-1}$ and $1-X^{h+1}$ are also divisible by $1+X$ so that we may write
\[
\mathcal{P}_h(X) = \dfrac{\mathcal{G}_h(X)}{(1+X)^2\,(1-X)^2\,\dfrac{1-X^{h-1}}{1-X^2}\,\dfrac{1-X^{h+1}}{1-X^2}}.
\]
Since $(1+X)^2$, $(1-X)^2$, $\dfrac{1-X^{h-1}}{1-X^2}$ and $\dfrac{1-X^{h+1}}{1-X^2}$ are
pairwise coprime, there exist $Q,R,A_0,B_0\in\Q[X]$ such that $\deg(Q)<2$, $\deg(R)<2$,
$\deg(A_0)<h-3$, $\deg(B_0)<h-1$ and
\[
\mathcal{P}_h(X) = \dfrac{Q(X)}{(1+X)^2}+\dfrac{R(X)}{(1-X)^2}+\dfrac{A_0(X)}{\dfrac{1-X^{h-1}}{1-X^2}}+\dfrac{B_0(X)}{\frac{1-X^{h+1}}{1-X^2}}.
\]
We get the result with $A(X)=(1-X^2)A_0(X)$ and $B(X)=(1-X^2)B_0(X)$.

\end{proof}

\begin{lemma}\label{lem::pal-even-odd-0}
For all $L,h\in \N$ such that $L$ is even and $h$ is odd, $p(L,h)=0$.
\end{lemma}
\begin{proof}
By definition of $p(L,h)$, the result is obvious if $L=0$. Also, it is sufficient to prove
the result for $0\leq h\leq L$ because $h$ and $h\bmod L$ have the same parity if $L$ is
even. In this case, $p(L,h)$ is exactly the number of balanced palindromes of length $L$
and height $h$. Let $w$ be a palindrome of length $L$. If $L$ 
is even then $w=u\Tilde{u}$ for some $u\in\{0,1\}^*$ and $|w|_1 = 2|u|_1$. Hence, there exist
no palindrome of even length and odd height.
\end{proof}

We are now ready to prove the main theorem of this section.\\

\begin{proof}[of Theorem~\ref{th::pLh-linear}]

The proof is similar to the proof of Theorem~\ref{th::sgen-quadratic}. In
Equations~\ref{eqn::pgen-partfrac-even} and~\ref{eqn::pgen-partfrac-odd}, we write

$$    R(X) = \alpha+\beta\,(1-X), \quad
    Q(X) = \alpha'+\beta'\,(1+X)$$
    
    $$A(X) = \sum_{k=0}^{h-2} a_k X^k,\quad
    B(X) = \sum_{k=0}^{h} b_k X^k
   $$
and we get for all $h\geq2$:
\begin{itemize}
\item If $h$ is even then
  \[
  \mathcal{P}_h(X) =  \dfrac{\alpha}{(1-X)^2}+\dfrac{\beta}{1-X}+\dfrac{A(X)}{1-X^{h-1}}+\dfrac{B(X)}{1-X^{h+1}}
  \]
  and the series expansion of $\mathcal{P}_h(X)$ is
  \begin{eqnarray*}
    \mathcal{P}_h(X) &=& \sum_{L\geq0} (\alpha\,(L+1)+\beta)\,X^L
    + \sum_{n\geq0} A(X)\,X^{n(h-1)} + \sum_{n\geq0} B(X)\,X^{n(h+1)}\\
    &=& \sum_{L\geq0} (\alpha\,(L+1)+\beta)\,X^L
    + \sum_{n\geq0}  \sum_{k=0}^{h-2} a_k X^{n(h-1)+k} + \sum_{n\geq0} \sum_{k=0}^{h} b_k X^{n(h+1)+k}\\
    &=& \sum_{L\geq0} (\alpha\,L+\alpha+\beta + a_{L\bmod(h-1)} + b_{L\bmod(h+1)})\,X^L
  \end{eqnarray*}
  Hence
  \[
  \forall L \geq 0,~p(L,h) = \alpha\,L+\alpha+\beta + a_{L\bmod(h-1)} + b_{L\bmod(h+1)}
  \]

\item If $h$ is odd then
  \[
  \mathcal{P}_h(X) = \dfrac{\alpha'}{(1+X)^2}+\dfrac{\beta'}{1+X}+\dfrac{\alpha}{(1-X)^2}+\dfrac{\beta}{1-X}+\dfrac{A(X)}{1-X^{h-1}}+\dfrac{B(X)}{1-X^{h+1}}
  \]
  and its series expansion is
  \[
  \mathcal{P}_h(X) = \sum_{L\geq0} (\alpha'\,(-1)^L\,(L+1) + \beta'\,(-1)^L + \alpha\,(L+1) + \beta + a_{L\bmod(h-1)} + b_{L\bmod(h+1)})X^L.
  \]
  Hence
  \[
  \forall L \geq 0,~p(L,h) = \alpha'\,(-1)^L\,(L+1) + \beta'\,(-1)^L + \alpha\,(L+1) + \beta + a_{L\bmod(h-1)} + b_{L\bmod(h+1)}.
  \]
\end{itemize}
Considering the Taylor expansion of $(1-X)^2\mathcal{P}_h(X)$ at $X=1$ and, in case $h$ is even,
$(1+X)^2\mathcal{P}_h(X)$ at $X=-1$, we get
\[
\alpha = \dfrac{\mathcal{G}_h(1)}{h^2-1} 
       = \dfrac{1}{h^2-1}\sum_{r=0}^{h-2}p(h-1,r) 
       = \dfrac{1}{h^2-1}\sum_{i=1}^{\left\lceil \frac{h-1}{2} \right\rceil}\varphi(h+1-2i)
\]
and, if $h$ is odd,
\[
\alpha' = \dfrac{\mathcal{G}_h(-1)}{h^2-1}
 = \dfrac{(-1)^h}{h^2-1}\sum_{r=0}^{h-2}(-1)^r p(h-1,r)
 = \dfrac{-1}{h^2-1}\sum_{r=0}^{h-2} p(h-1,r)
\]
where the last equality is deduced from Lemma~\ref{lem::pal-even-odd-0}. Hence $\alpha'=-\alpha$.

Finally, we get the result with $u_i = \alpha+\beta+ a_i$ if
$h$ is even and $u_i = (1- (-1)^i)\alpha+\beta+(-1)^i\beta'+ a_i$ if $h$ is odd and
$v_i = b_i$.

\end{proof}

{\bf Acknowledgements}:
The authors would like to thank Julien Cassaigne for useful discussions and comments.

\bibliographystyle{alpha}
\bibliography{Biblio}

\end{document}